\documentclass[11pt]{article}

\usepackage{amsfonts}
\usepackage{amsmath}
\usepackage{amsthm}
\usepackage{amssymb}
\usepackage{esint}
\usepackage{bbm}
\usepackage{fullpage}
\usepackage{hyperref}
\usepackage{url}
\usepackage[sort,numbers]{natbib}

\theoremstyle{plain} 
\newtheorem{theorem}{Theorem}[section]
\newtheorem{corollary}[theorem]{Corollary}
\newtheorem{lemma}[theorem]{Lemma}

\newtheorem{proposition}[theorem]{Proposition}

\newtheorem*{lemma*}{Lemma}
\newtheorem*{claim*}{Claim}
\newtheorem*{theorem*}{Theorem}
\newtheorem*{proposition*}{Proposition}
\newtheorem*{problemstatement*}{Problem}
\newtheorem*{corollary*}{Corollary}

\theoremstyle{definition} 
\newtheorem{definition}[theorem]{Definition}

\newtheorem*{notation*}{Notation}
\newtheorem*{definition*}{Definition}
\newtheorem*{conjecture*}{Conjecture}
\newtheorem*{example*}{Example}

\theoremstyle{remark} 
\newtheorem{remark}[theorem]{Remark}

\begin{document}
\title{A probabilistic look at the infinite hat-guessing game}
\author{Nathaniel Eldredge\footnote{Department of Mathematical Sciences, University of Northern Colorado; Greeley, Colorado, USA.  \mbox{Email: \href{mailto:neldredge@unco.edu}{neldredge@unco.edu}}}}
\maketitle

\begin{abstract}
  In this article, we look at a hat-guessing game, in which each player must guess the color of their own hat while only seeing the hats of the other players.  We focus on the case of two hat colors and a countably infinite number of players.   By strategizing in advance, the players can, in some ways, do much better than random guessing; using the axiom of choice, they can in fact achieve highly counter-intuitive success.  We review some of these results.  Then, we use tools from probability to obtain bounds on how successful a strategy can be under a measurability hypothesis, in terms of the asymptotic density of the set of correctly guessing players.  As we discuss, this illustrates that the full axiom of choice is truly necessary for the counter-intuitively successful strategies, and that there is a wide gap between what can be achieved with and without choice.
\end{abstract}

\section{Introduction}

Consider the following hat-guessing game: some number of players (possibly infinite) gather in a large room.  Each player is assigned, by fair coin flip,  either a black or a white hat, which is placed on their head.    They cannot see the color of their own hat, but can see the hats of all the other players.  Then each player is asked to guess the color of their own hat.  The players do not hear each other's guesses, and cannot communicate in any way after the hats have been distributed.  However, they are allowed to strategize together before the game starts, with the aim of maximizing the fraction of players who guess correctly.

On its face, it may seem that the players cannot do better than random guessing, as the information available to each player (the colors of the other players' hats) appears to be irrelevant to what the player must guess (the color of their own hat).   However, the true state of affairs is more subtle.  The interesting feature of this game is that, although indeed no \emph{individual} player can do better than a random guess, they can make use of their shared information to introduce correlations between their guesses, and this can lead to better outcomes for the group as a whole.

Our main interest here will be the case of (countably) infinitely many players, and on connections with the axiom of choice ($\mathsf{AC}$).  In this setting, it has been shown \cite{ht-intel} that there exist strategies for the game that are shockingly successful, in that they ensure the players guess correctly to an extent that intuitively seems impossible.  For instance, using $\mathsf{AC}$ one can produce a strategy that guarantees that \emph{all but finitely many} players will correctly guess the colors of their own hats!  We shall review a few of these results in Section \ref{sec:AC}, and by the end of this paper, you may agree with the author that they should rank with the Banach--Tarski paradox among the most counter-intuitive consequences of the choice axiom.

On the other hand, although using $\mathsf{AC}$ we can prove the \emph{existence} of such strategies, we cannot in any sense describe them explicitly.  Indeed, it is known that these strategies correspond to subsets of $\mathbb{R}$ which are  not Lebesgue measurable and lack the property of Baire.  This means, as we discuss in Section \ref{sec:alt}, that we really need the full strength of the axiom of choice to produce them, so that they are, in some sense, inherently non-constructive.    But it also means that, if we were to reject $\mathsf{AC}$ and adopt certain alternative axioms, we could live in a consistent mathematical universe where such strategies simply do not exist, hewing closer to our intuition about the game.  

In Section \ref{sec:meas}, we focus on measurable strategies.  This allows us to bring probability theory to bear, which we think provides a helpful perspective on the game.  The concept of independence is particularly useful.  It seems to us that a probabilistic approach to the hat-guessing game is quite natural, but we have not seen this view taken explicitly in the existing literature, which has largely been based on combinatorics and/or descriptive set theory.

The main new result of this paper is Theorem \ref{main},  obtaining almost sure bounds on the asymptotic density of the set of players who guess correctly.  This shows, in a rather strong sense, that the success of the $\mathsf{AC}$ strategies is far beyond what can be attained by measurable strategies, and in particular, far beyond what can be achieved without $\mathsf{AC}$.

In Section \ref{sec:explicit-strategies}, we further explore what measurable strategies can achieve in terms of asymptotic density.  Even within the bounds established by Theorem \ref{main}, some surprising results remain possible.  We also look briefly at other variants on the game in Section \ref{sec:variants}.

The author would like to thank Clinton Conley, Elliot Glazer, and Charles Wang for helpful discussions.  Thanks are also due to the students in MATH 6710 at Cornell University in Fall 2012, for their patience in working on a homework exercise, assigned by the author, that eventually led to this paper.

\section{Finite hat-guessing}\label{sec:finite}

As a warm-up, we consider the hat-guessing game with a \emph{finite} number of players $n$.  Although elementary, it already has some features of interest.

For $1 \le i \le n$, let $\xi_i$ be a random variable (all defined on some fixed probability space $(\Omega, \mathcal{F}, \mathbb{P})$) representing the color of the hat assigned to player $i$, taking the values $0$ and $1$ for ``black'' and ``white'' respectively.  Since we are thinking of the hats as assigned by fair coin flips, we'll take $\{\xi_i : 1 \le i \le n\}$ to be iid, with $\mathbb{P}(\xi_i = 0) = \mathbb{P}(\xi_i = 1) = \frac{1}{2}$.

Now let $X_i$ be the color guessed by player $i$, according to whatever strategy has been agreed by the players in advance.  To enforce the rule that players cannot see their own hats, for each $i$ define the $\sigma$-field $\mathcal{G}_i = \sigma(\xi_1, \dots, \xi_{i-1}, \xi_{i+1}, \dots, \xi_n)$ representing the information available to player $i$.  We then require that $X_i$ is a $\mathcal{G}_i$-measurable random variable, which, following a common abuse of notation, we abbreviate by $X_i \in \mathcal{G}_i$.  This is equivalent to requiring that each $X_i$ can be written as $X_i = f_i(\xi_1, \dots, \xi_{i-1}, \xi_{i+1}, \dots, \xi_n)$ for some function $f_i : \{0,1\}^{n-1} \to \{0,1\}$.

Thus a \emph{strategy} in this game may be represented simply by a family of $\{0,1\}$-valued random variables $X_1, \dots, X_n$, satisfying the above condition that $X_i \in \mathcal{G}_i$ for each $i$; or, equivalently, by the corresponding functions $f_i$.

We begin by noting that $X_i \in \mathcal{G}_i$ implies that $X_i$ is independent of $\xi_i$.  This easily implies that each player, considered individually, can do no better (nor worse) than 50--50.

\begin{proposition}[cf.~{\cite[Lemma 2]{ht-intel}, \cite[Lemma 2.2.1]{ht-book}}]\label{fifty-fifty}
  For each $i$, we have $\mathbb{P}(X_i = \xi_i) = \frac{1}{2}$.
\end{proposition}

\begin{proof}
  \begin{align*}
    \mathbb{P}(X_i = \xi_i) &= \mathbb{P}(\{X_i = 0\} \cap \{\xi_1 = 0\}) + \mathbb{P} (\{X_i = 1\} \cap \{\xi_1 = 1\}) \\
    &= \mathbb{P}(X_i = 0) \mathbb{P}(\xi_i = 0) + \mathbb{P}(X_i = 1) \mathbb{P}(\xi_i=1)  && \text{(by independence)} \\
    &=  \mathbb{P}(X_i=0) \cdot \frac{1}{2} + \mathbb{P}(X_i = 1) \cdot \frac{1}{2} \\
    &= \frac{1}{2}(\mathbb{P}(X_i = 0) + \mathbb{P}(X_i = 1)) = \frac{1}{2}.
  \end{align*}
\end{proof}

Note that this assumes nothing about the distribution of $X_i$, and applies even if they have adopted a strategy that biases their guesses toward a particular color.

So if we let $Z_i = \mathbbm{1}_{\{X_i = \xi_i\}}$ be the indicator of the event that player $i$ guesses correctly (that is, $1$ if they guess correctly and $0$ if they do not), we have $\mathbb{P}(Z_i = 0) = \mathbb{P}(Z_i = 1) = \frac{1}{2}$.   It is important to note that the $Z_i$ need not be independent of each other.  Even so, letting $\overline{Z} = \frac{1}{n} (Z_1 + \dots + Z_n)$ be the fraction of players who guess correctly, by linearity of expectation we still have:
\begin{proposition}\label{EF-half}
  $\mathbb{E}[\overline{Z}] = \frac{1}{2}$.
\end{proposition}
We see, then, that no clever strategy can improve the chances for any given player to guess correctly, nor can it improve the average number of correct guesses.  What a strategy \emph{can} do is to create correlations between those correct guesses.  Before seeing some examples, let us first find some bounds on what such strategies could accomplish.

Recalling that Markov's inequality gives us $\mathbb{P}(\overline{Z} \ge \alpha) \le \frac{1}{\alpha} \mathbb{E}[\overline{Z}] = \frac{1}{2 \alpha}$ for all $\alpha > 0$, taking  $\alpha=1$ we have:
\begin{proposition}\label{F1}
  $\mathbb{P}(\overline{Z}=1) \le \frac{1}{2}$.  That is, no strategy can provide more than a 50\% chance that all players guess correctly.
\end{proposition}

And taking arbitrary $\alpha > \frac{1}{2}$ we get:
\begin{proposition}[cf.~{\cite[Theorem 3]{ht-intel}},{\cite[Theorem 2.3.1]{ht-book}} ]\label{Fa}
  $\mathbb{P}(\overline{Z} \ge \alpha) < 1$ for all $\alpha > \frac{1}{2}$.  That is, no strategy can guarantee that more than half of players guess correctly.
\end{proposition}

At first glance, these bounds may seem very far from sharp.  For instance, a naive ``guess randomly'' strategy, in which the $X_i$ are mutually independent (or simply constant), has $\mathbb{P}(\overline{Z}=1) = \frac{1}{2^n}$.
  Moreover, under this strategy we have $\mathbb{P}(\overline{Z} \ge \alpha) < 1$ for all $\alpha > 0$: there is no guarantee that \emph{any} players guess correctly.  They could all guess wrong, which happens with probability $\frac{1}{2^n}$.

However, with more clever strategies we can do much better.  In fact, both of the above bounds (Propositions \ref{F1} and \ref{Fa}) are in fact sharp!

\begin{definition}[cf.~\cite{winkler}]\label{even-odd}
  The \emph{even-odd strategy} is defined by
  \begin{equation*}
    X_i = (\xi_1 + \dots + \xi_{i-1} + \xi_{i+1} + \dots + \xi_n) \bmod 2.
  \end{equation*}
  That is, each player guesses ``white'' if they see an odd number of white hats on the other players, and ``black'' otherwise.  Equivalently, each player guesses the color that their hat would be if the total number of white hats were even.
\end{definition}

\begin{proposition}\label{even-odd-result}
  Under the even-odd strategy, $P(\overline{Z}=1) = \frac{1}{2}$.
\end{proposition}

\begin{proof}
  If the total number of white hats is indeed even, then all players guess correctly.  It is clear from symmetry that this occurs with probability $\frac{1}{2}$.  On the other hand, if the total number of white hats is odd, then every player guesses incorrectly, so that $\overline{Z}=0$; this is inevitable in view of the expectation $\mathbb{E}[\overline{Z}] = \frac{1}{2}$ (Proposition \ref{EF-half}).
\end{proof}

Now we turn to the bound of Proposition \ref{Fa}.

\begin{definition}[see \cite{winkler}]\label{pairs}
  Suppose that $n$ is even.  The \emph{pairs} strategy is defined by
  \begin{equation*}
    X_{2j-1} = \xi_{2j}, \qquad 
    X_{2j} = 1 - \xi_{2j-1}
  \end{equation*}
  for each $j = 1, \dots, \frac{n}{2}$,  That is, the players group themselves into pairs: Player 1 with Player 2, Player 3 with Player 4, and so on.  Within the first pair, Player 1 guesses that his hat color is the same color as the hat he sees on Player 2, and Player 2 guesses that her hat is the \emph{opposite} of Player 1's.  The other pairs play similarly.
\end{definition}

\begin{proposition}
  Under the pairs strategy, $\overline{Z} = \frac{1}{2}$ surely (i.e. for every possible hat assignment).
\end{proposition}

\begin{proof}
  Consider Players 1 and 2.  If their hats are the same color, then Player 1 guesses correctly and Player 2 guesses incorrectly.  If their hats are opposite colors, then the reverse is true.  In either case, exactly one player of the pair guesses correctly, and this same logic applies to every other pair.  So we are guaranteed that exactly half of the players guess correctly.
\end{proof}

One way to think about these results is that, although Proposition \ref{EF-half} tells us that on average, half the players must inevitably guess wrong, we still have the freedom to create correlations or anticorrelations between the correctness of their guesses.  Thus we can ensure that correct guesses are either clumped together, so that we can get a lot of them at once, or spread apart, so that we never have too few at the same time.

\section{Infinite hat-guessing with the axiom of choice}\label{sec:AC}

We turn now to the case of an \emph{infinite} number of players.  Specifically, we suppose that there are countably many players, indexed by the positive integers $\mathbb{N} = \{1,2,3, \dots\}$.  As before, $\xi_i$ denotes the hat color given to player $i$.

In this section, we will work with the axiom of choice ($\mathsf{AC}$), and work with strategies that need not be measurable.  As such, we cannot really do any probability theory here.  However, notationally we will still let $X_i$ denote the guess of player $i$, and require that $X_i = f_i(\xi_1, \dots, \xi_{i-1}, \xi_{i+1}, \dots)$ where $f_i : \{0,1\}^{\mathbb{N} \setminus \{i\}} \to \{0,1\}$ is an \emph{arbitrary} function (not necessarily measurable in any sense).  A \emph{strategy} is thus a sequence of such functions $f_i$.  In this setting, we can no longer speak in general of the probability of a guess being correct, as the set of hat assignments resulting in a correct guess need not be measurable, and so may not have a well-defined probability.  But we can still talk about results that occur ``surely''; i.e. for every possible hat assignment in $\{0,1\}^{\mathbb{N}}$.

In the case of finitely many players, we showed that it is impossible to guarantee that more than half of the players guess correctly (Proposition \ref{Fa}).  In the infinite case, using the axiom of choice, this limitation can be demolished.

\begin{theorem}[Attributed to Yuval Gabay and Michael O'Connor; see {\cite[Theorem 4]{ht-intel}}]
  Assuming the axiom of choice, there exists a strategy under which, surely, all but finitely many players guess correctly.
\end{theorem}

\begin{proof}
  Consider the relation $\sim$ on the set $\{0,1\}^{\mathbb{N}}$ of $\{0,1\}$-valued sequences, defined by $\omega \sim \omega'$ if $\omega_i = \omega'_i$ for all but finitely many indices $i$.  It is clear that $\sim$ is an equivalence relation and thus partitions $\{0,1\}^{\mathbb{N}}$ into equivalence classes.  Using the axiom of choice, form a set $A \subseteq \{0,1\}^{\mathbb{N}}$ which contains exactly one element of each equivalence class.

  Now for each $i$, define the strategy $f_i$ of Player $i$ by setting
  $f_i(\xi_1, \dots, \xi_{i-1}, \xi_{i+1}, \dots) = \omega_i$, where $\omega$ is the unique element of $A$ with $\omega \sim (\xi_1, \dots, \xi_{i-1}, \ast, \xi_{i+1}, \dots)$.  Here $\ast$ could be taken to be either $0$ or $1$ without changing the resulting $\omega$.

  In words, let $\xi = (\xi_1, \xi_2, \dots)$ denote the sequence of hats assigned to the players.  Since each player $i$ sees all but one entry of $\xi$, they can correctly identify its equivalence class, and thus the agreed-upon representative $\omega \in A$ of that class.  Player $i$ then guesses that their hat is of color $\omega_i$.  By construction, $\omega \sim \xi$, so $\omega_i = \xi_i$ for all but finitely many $i$; that is, all but finitely many players guess correctly.
\end{proof}

To the author's mind, the Gabay--O'Connor theorem is an astounding result, comparable to the  Banach--Tarski paradox among the most perplexing consequences of the axiom of choice.  Intuitively the game with infinitely many players ought to be, if anything, \emph{harder} to win than with finitely many players.  So how is it possible that we can do so much better in the infinite case?  In the finite case, we easily proved (Proposition \ref{F1}) that the maximum fraction of players who can be guaranteed to guess right is $\frac{1}{2}$, yet in the infinite case, somehow we are able to push this fraction to $1$.  It seems patently absurd.

Our main result in Section \ref{sec:meas}, together with the discussion in Section \ref{sec:alt}, will show that rejecting the axiom of choice would restore order and reason to the universe, or at least to this hat-guessing game, and re-impose an upper bound of $\frac{1}{2}$, in an appropriate sense.  This rules out all such impossibly successful strategies in what we think is a rather strong quantitative manner.

\begin{remark}
  The idea of the Gabay--O'Connor theorem actually goes back much further, to an American Mathematical \textsc{Monthly} problem from 1966 proposed (and solved) by F.~Galvin \cite{galvin-silverman-solution, thorp-solution}, in the following guise: find a function $f : \mathbb{R}^{\mathbb{N}} \to \mathbb{R}$ such that for all sequences $(x_1, x_2, \dots) \in \mathbb{R}^{\mathbb{N}}$, for all but finitely many $n$ we have $x_n = f(x_{n+1}, x_{n+2}, \dots)$.  This can be seen as a hat-guessing game with ``limited visibility'', in that each player can only see the hats of higher-numbered players, and with uncountably many hat colors; perhaps surprisingly, the construction described above goes through without change in this apparently more difficult setting.  (See Section \ref{sec:limited-vis} below for more on the limited visibility game.)

  Interestingly, the first solution published in the \textsc{Monthly} \cite{galvin-silverman-solution} was flawed: in terms of hat guessing, it purported to present a strategy that ensured that \emph{at most one} player guesses incorrectly.  In fact, no such strategy is possible.  The guarantee of the Gabay--O'Connor strategy that only \emph{finitely many} players guess incorrectly is in some sense best possible: it cannot be improved to \emph{boundedly many}.  Indeed, let a strategy be fixed, and for any $n$, let us restrict our attention to the $2^n$ hat assignments in which the hat colors of players $1, \dots, n$ are unrestricted, but all remaining players receive black hats.  Then players $1, \dots, n$ are effectively just playing a finite version of the hat game among themselves; there are no measurability concerns because only a finite number of outcomes are under consideration.  We then see, just as in Proposition \ref{Fa}, that there must be at least one of these $2^n$ assignments in which at least $\frac{n}{2}$ players guess incorrectly.  
\end{remark}

In a different direction, if instead of ``spreading'' successful guesses we want to ``clump'' them, we have the following result:

\begin{theorem}[Attributed to Hendrik Lenstra; see {\cite[Theorem 5]{ht-intel}}]\label{evenodd-ac}
  Assuming the axiom of choice, there exists a strategy under which, surely, either all players guess correctly, or all players guess incorrectly.
\end{theorem}

This result seems less problematic on its face than the Gabay--O'Connor theorem.  After all, it does not seem to violate the spirit of Proposition \ref{Fa}, and indeed we found a strategy in the finite case which achieves just the same result: namely, the even-odd strategy of Definition \ref{even-odd}.  However, if we try to extend the even-odd strategy to the infinite case, we see a problem: it somehow requires us to be able to classify \emph{infinite} subsets of $\mathbb{N}$ as having an ``even'' or ``odd'' number of elements, in such a way that adding or removing one element from a set should switch its parity.  There is certainly no obvious elementary way to do so.  Still, we can do it using the axiom of choice.

\begin{proof}
  View the set of hat colors $\{0,1\}$ as the finite field $\mathbb{F}_2$,  and view the set of possible hat assignments $\{0,1\}^{\mathbb{N}}$ as the infinite-dimensional vector space $\mathbb{F}_2^{\mathbb{N}}$ over $\mathbb{F}_2$.  For each $i$, let $\delta^i \in \mathbb{F}_2^{\mathbb{N}}$ be the sequence defined by $\delta^i_i=1$ and $\delta^i_j = 0$ for $i \ne j$.  Then $\{\delta_i : i \in \mathbb{N}\}$ is a linearly independent set in $\mathbb{F}_2^{\mathbb{N}}$.  Using Zorn's lemma, extend $\{\delta^i\}$ to a basis $B$ for $\mathbb{F}_2^{\mathbb{N}}$, in the sense that every element of $\mathbb{F}_2^{\mathbb{N}}$ can be written uniquely as a \emph{finite} linear combination of elements of $B$ (a so-called Hamel basis).  Fix any function $g : B \to \mathbb{F}_2$ satisfying $g(\delta^i) = 1$ for all $i$; its values on the other elements of $B$ can be anything.  Since $B$ is a basis, $g$ extends uniquely to a linear functional $g : \mathbb{F}_2^{\mathbb{N}} \to \mathbb{F}_2$.

  Now $g(\omega)$ gives us, as desired, a notion of ``parity'' for the possibly infinite  set $\{ i : \omega(i)=1\}$. Indeed, the empty set has even parity (as $g(0)=0$), and the fact that $g$ is $\mathbb{F}_2$-linear and $g(\delta^i) = 1$ implies that adding or removing one element from a set will switch its parity.  In particular, $g$ extends the usual notion of parity for finite sets, though this extension is highly non-unique and in no way canonical.

  Now to define Lenstra's strategy, let
  \begin{equation*}
    f_i(\xi_1, \dots, \xi_{i-1}, \xi_{i+1}, \dots) = g(\xi_1, \dots, \xi_{i-1}, 0, \xi_{i+1}, \dots).
  \end{equation*}
  Just as in the even-odd strategy of Definition \ref{even-odd}, this represents the ``parity'' of the white hats seen by Player $i$, or equivalently, the color their own hat would be if the set of all white hats had even parity ($g(\xi) = 0$).  So if $\xi$ does indeed satisfy $g(\xi)=0$ then all players guess correctly.    On the other hand, if instead we have $g(\xi) = 1$ (odd parity), then by the same token, all players guess incorrectly.
\end{proof}

There are other ways to use $\mathsf{AC}$ to produce parity functions of this kind; see \cite{glr}.

The Gabay--O'Connor and Lenstra strategies show, in a rather vivid way, that the use of the axiom of choice allows the players to succeed at the hat-guessing game to an extent that our intuition would suggest was impossible.  In Section \ref{sec:meas}, we show that \emph{measurable} strategies cannot achieve any such degree of success, and behave much more in line with our intuition.    It follows from this, as we discuss in Section \ref{sec:alt}, that counter-intuitively successful strategies must inherently make use of the ``full strength'' of the axiom of choice, and in that sense, cannot be made ``explicit.''  Indeed, we will see that it is possible to reject $\mathsf{AC}$ and work instead in an alternative, yet still consistent, axiomatic system in which these overly successful strategies simply do not exist.  

\section{Measurable hat guessing via probability theory}\label{sec:meas}

In this section, we continue to study the hat-guessing game with two hat colors and countably many players, following the same setup and notation as in Section \ref{sec:AC}.  However, we shall now impose the restriction that the players may only use \emph{measurable} strategies.  This allows us to make use of probability theory once more, and as we will see, greatly limits the success that the players can achieve.

Specifically, let $\{0,1\}^{\mathbb{N}}$ be equipped with the Borel $\sigma$-field induced by its product topology.  As the product topology is homeomorphic to the Cantor set, which in particular is Polish, this gives us a standard Borel space.  Now let $\mu$ denote the ``coin-flipping'' measure on $\{0,1\}^{\mathbb{N}}$; i.e., the unique Borel probability measure such that the coordinate maps represent iid random variables which are 0 or 1 with probability $\frac{1}{2}$.  We require, for the duration of this section, that each strategy function $f_i : \{0,1\}^{\mathbb{N} \setminus \{i\}} \to \{0,1\}$ is $\mu$-measurable, so that it is $\mu$-a.e. equal to a Borel function.

In particular, if $\{ \xi_i : i = 1, 2, \dots\}$ are iid random variables with $\mathbb{P}(\xi_i = 0) = \mathbb{P}(\xi_1 = 1) = \frac{1}{2}$, defined on some complete probability space $(\Omega, \mathcal{F}, \mathbb{P})$, then each $X_i =  f_i(\xi_1, \dots, \xi_{i-1}, \xi_{i+1}, \dots)$, representing the guess made by player $i$, is actually a random variable.  It is measurable with respect to the $\sigma$-field $\mathcal{G}_i = \sigma(\xi_1, \dots, \xi_{i-1}, \xi_{i+1}, \dots)$ representing the information available to player $i$.  As before, we let $Z_i = \mathbbm{1}_{\{X_i = \xi_i\}}$ be the indicator of the event that Player $i$ guesses correctly, and Proposition \ref{fifty-fifty} shows $\mathbb{E}[Z_i] = \frac{1}{2}$ for all $i$.  We also let $S_n = \sum_{i=1}^n Z_i$ be the number of correct guesses among the first $n$ players, and $\overline{Z}_n = S_n / n$ be the fraction of correct guesses; note $\mathbb{E} [\overline{Z}_n] = \frac{1}{2}$ for all $n$, by linearity of expectation.  We shall focus our attention on the lower asymptotic density $L = \liminf_{n \to \infty} \overline{Z}_n$ of the (random) set of players who guess correctly, and likewise the upper asymptotic density $U = \limsup_{n \to \infty} \overline{Z}_n$.

We can very quickly rule out  any analogue of the Gabay--O'Connor strategy in this setting, because if it were guaranteed that all but finitely many players guess correctly, then $L$ would be the constant $1$.  In particular, we would have $\mathbb{E}[L]=1$.  But Fatou's lemma tells us that $\mathbb{E}[L] \le \liminf_{n \to \infty} \mathbb{E}[\overline{Z}_n] = \frac{1}{2}$.

However, this leaves open the possibility of a strategy that could provide a \emph{positive} probability for all but finitely players to guess right; indeed, $\mathbb{E}[L] \le \frac{1}{2}$ is consistent with that probability being as high as $\frac{1}{2}$.  This is still hard to accept intuitively.  Moreover, $\mathbb{E}[L] \le \frac{1}{2}$ doesn't rule out a Lenstra-type  strategy where either all players guess right or all players guess wrong.

To achieve this, we will prove a stronger result, that excludes both possibilities at once.

\begin{theorem}\label{main}
 $\mathbb{P}(L \le \frac{1}{2} \le U) = 1$.  That is, for every measurable strategy, almost surely, the set of players who guess correctly has lower density at most $\frac{1}{2}$, and upper density at least $\frac{1}{2}$.
\end{theorem}

In particular, we have that almost surely, infinitely many players guess right and infinitely many players guess wrong.

The proof is based on a basic principle about measurable functions, which, following the language of Littlewood's famous three principles \cite[p.~26]{littlewood}, we might state as: ``A measurable function on $\{0,1\}^{\mathbb{N}}$ \emph{almost} depends only on finitely many coordinates.''  Of course, a strategy which was based only on observing finitely many hats could not possibly achieve the success of the Gabay--O'Connor or Lenstra strategies, which clearly must make essential use of all the hat information observed by the players.  If a measurable strategy must be ``close'' to a strategy of the former kind, then it stands to reason that it cannot achieve such success either.
  
  For us, this principle manifests via the following result about conditional expectations on filtrations, which is usually proved as a corollary of the martingale convergence theorem.  For each $m$, let $\mathcal{F}_m = \sigma(\xi_1, \dots, \xi_m)$ be the $\sigma$-field generated by the first $m$ hat colors, and let $\mathcal{F}_\infty = \sigma(\mathcal{F}_m : m \ge 1) = \sigma(\xi_i : i \ge 1)$ be the $\sigma$-field generated by all of them.  Notice that all random variables under consideration are $\mathcal{F}_\infty$-measurable; this expresses the fact that, once the strategy has been chosen, the results of the game are completely determined by the hat colors.
  
  \begin{theorem}[See for instance {\cite[Theorem 4.6.8]{durrett-pte5}}]\label{martingale-converge}
    If $X$ is any integrable $\mathcal{F}_\infty$-measurable random variable, then $\mathbb{E}[X \mid \mathcal{F}_m] \to X$ as $m \to \infty$, almost surely and in $L^1$.
  \end{theorem}

  We can view $\mathbb{E}[X \mid \mathcal{F}_m] \to X$ as the best estimate we could make (in the sense of mean squared error) of $X$, if all we can see are the first $n$ hats.  If we identify $\{0,1\}^{\mathbb{N}}$ with $[0,1]$ via binary expansions, then Theorem \ref{martingale-converge} can also be derived as a special case of the Lebesgue differentiation theorem.

We begin the proof of Theorem \ref{main} by noting that, if we can observe all hats except the $i$th, it not only gives us no information about the $i$th hat, it also tells us nothing about Player $i$'s chances of guessing correctly.  

\begin{lemma}\label{Z-G-indep}
  $Z_i$ is independent of $\mathcal{G}_i$.
\end{lemma}

\begin{proof}
  Write $Z_i = X_i \xi_i + (1-X_i)(1-\xi_i)$.  Since $X_i \in \mathcal{G}_i$ and $\xi_i$ is independent of $\mathcal{G}_i$, we have by properties of conditional expectation:
  \begin{align*}
    \mathbb{E}[Z_i \mid \mathcal{G}_i] &= X_i \mathbb{E}[\xi_i \mid \mathcal{G}_i] + (1-X_i) \mathbb{E}[1-\xi_i \mid \mathcal{G}_i] \\
      &= X_i \mathbb{E}[\xi_i] + (1-X_i) \mathbb{E}[1-\xi_i] \\
      &= \frac{1}{2} X_i + \frac{1}{2} (1-X_i) \\
      &= \frac{1}{2}.
  \end{align*}
  Since $Z_i$ is 0-1-valued, we have $\mathbb{P}(Z_i = 1 \mid \mathcal{G}_i) = \mathbb{E}[Z_i \mid \mathcal{G}_i] = \frac{1}{2}$, and so $\mathbb{P}(Z_i = 0 \mid \mathcal{G}_i) = \frac{1}{2}$ as well.  Thus the conditional distribution of $Z_i$ given $\mathcal{G}_i$ is the same as its unconditional distribution, which implies independence.
\end{proof}

We can now complete the proof of the main theorem.

\begin{proof}[Proof of Theorem \ref{main}]
  For any $m < i$, we have $\mathcal{F}_m \subseteq \mathcal{G}_i$.  Thus Lemma \ref{Z-G-indep} implies that $Z_i$ is independent of $\mathcal{F}_m$, so that $\mathbb{E}[Z_i \mid \mathcal{F}_m] = \mathbb{E}[Z_i] = \frac{1}{2}$.  
  
  Now for any $m < k$ we have, by linearity,
  \begin{align*}
    \mathbb{E}[\overline{Z}_k \mid \mathcal{F}_m] &= \frac{1}{k} \sum_{i=1}^m \mathbb{E}[Z_i \mid \mathcal{F}_m] + \frac{1}{k} \sum_{i=m+1}^k \mathbb{E}[Z_i \mid \mathcal{F}_m] \\
    &= \frac{1}{k} \sum_{i=1}^m \mathbb{E}[Z_i \mid \mathcal{F}_m] + \frac{1}{2} \frac{k-m}{k} \\
    &\le \frac{m}{k} + \frac{1}{2}
  \end{align*}
  because $Z_i \le 1$ implies $\mathbb{E}[Z_i \mid \mathcal{F}_m] \le 1$ as well.  
  Then by the conditional Fatou lemma, we have almost surely
  \begin{equation*}
    \mathbb{E}[L \mid \mathcal{F}_m] \le \liminf_{k \to \infty} \mathbb{E}[\overline{Z}_k \mid \mathcal{F}_m] \le \liminf_{k \to \infty} \left[ \frac{m}{k} + \frac{1}{2} \right] = \frac{1}{2}.
  \end{equation*}
  Since $\mathbb{E}[L \mid \mathcal{F}_m] \to L$ almost surely by Theorem \ref{martingale-converge}, we conclude that $L \le \frac{1}{2}$ almost surely as well.
\end{proof}

We can view Theorem \ref{main} as a kind of analogue of Proposition \ref{Fa}, which said that, with finitely many players, no strategy can guarantee that more than half of players guess correctly.  With infinitely many players, the statement is stronger: no (measurable) strategy can provide \emph{any} nonzero chance that more than half of players (in the sense of lower density) guess correctly.  This fits with our intuition that the infinite-player game ought to be strictly more difficult to win: adding more
players certainly does not seem to provide any advantages, and moreover, we lose access to elementary properties like parity.

In Section \ref{sec:explicit-strategies} below, we explore some explicit measurable strategies that illustrate the variety of results that remain possible within the bounds established by Theorem \ref{main}, and give some indication of their sharpness.

\begin{remark}
  Theorem \ref{main} is somewhat interesting in view of the fact that the lower and upper asymptotic densities of a subset of $\mathbb{N}$ are not invariant under permutations of $\mathbb{N}$; indeed, for any two subsets of $\mathbb{N}$ which are infinite and have infinite complements, there is a bijection of $\mathbb{N}$ which maps one to the other, even though their densities may be completely different.  So it is certainly possible for the players to achieve different asymptotic densities, in any given outcome, by simply renumbering themselves.    But Theorem \ref{main} shows that the players cannot \emph{statically} renumber themselves to achieve a lower density greater than $1/2$ with positive probability, since such a renumbering would just correspond to a different measurable strategy, subject to the same bounds.   Nor can they do so by \emph{dynamically} renumbering themselves after the game starts; if each player is able to determine their new number as a measurable function of the hats visible to them, then this too is just a different measurable strategy.
\end{remark}

\section{Alternatives to the axiom of choice}\label{sec:alt}

Theorem \ref{main} shows a sharp divide between measurable strategies, whose potential success it restricts,  and highly successful $\mathsf{AC}$-based strategies like Gabay--O'Connor.  For convenience in this section, we will say a strategy is either \emph{reasonably} or \emph{unreasonably successful} according to whether it does or does not satisfy the conclusion of Theorem \ref{main}, that $L \le \frac{1}{2} \le U$ almost surely.  By this definition, the Gabay--O'Connor and Lenstra strategies are unreasonably successful.  So would be a hypothetical ``weak Gabay--O'Connor'' strategy, promising for instance that all but finitely many composite-numbered players (a set of density 1, by the prime number theorem) would guess correctly; or one which did the same for square-free-numbered players (density $\frac{6}{\pi^2} \approx 0.608$
\cite[\S 18.6, Theorem 333]{hardy-wright}).
Theorem \ref{main} thus can be stated as ``measurable strategies can only be reasonably successful,'' or by contrapositive, ``all unreasonably effective strategies are non-measurable''.

We would like to remark at this point that it is already known that the Gabay--O'Connor and Lenstra strategies themselves must be non-measurable; see \cite{ht-intel} and its remark following Theorem 8 (which however omits details).  So, the effect of our Theorem \ref{main} is, in some sense, to enlarge the class of strategies known to be non-measurable.

Even within $\mathsf{ZFC}$, non-measurability has some ``practical'' implications.  For instance, Theorem \ref{main} implies in particular that Borel strategies can only be reasonably successful.  Now, let us imagine that the players hire a remote consultant to develop their strategy for them.  Having worked out the strategy, the consultant must communicate it to the players.  But suppose the consultant is on a limited-bandwidth communication link, and can only send and receive countably many bits of information.  If the strategy is to consist of Borel functions $f_i$, then there is a standard scheme, called \emph{Borel coding}, by which it can indeed be encoded and communicated using countably many bits, and then decoded and used by the players.  But by doing so, as we've shown, they would only be able to achieve \emph{reasonable} success.  If they want to achieve \emph{unreasonable} success,  then either they would need \emph{uncountably} many bits to communicate their strategy, or else they would have to work out a custom encoding scheme --- and then how do they communicate the encoding scheme?  

On the other hand, if we step back from $\mathsf{ZFC}$, then results such as Theorem \ref{main} tell us, roughly speaking, that $\mathsf{AC}$ is \emph{required} in order to prove that unreasonably successful strategies exist at all.  To keep this paper short, we give only an outline of the ideas involved. For further reading, we refer the reader to the last section of \cite{ht-intel}, as well as to \cite[Chapters 6 and 14]{schechter}.

Let $\mathsf{ZF}$ denote the usual Zermelo--Fraenkel axioms of set theory (without the axiom of choice).  Although $\mathsf{ZF}$ itself is inconveniently weak for doing general mathematics, we may append to it the \emph{axiom of dependent choice}  $\mathsf{DC}$, which, roughly speaking, allows one to make countably many choices in an iterative fashion.   This suffices for the standard results of analysis and probability on $\mathbb{R}$ and other Polish and standard Borel spaces; in particular, Theorem \ref{main} is a theorem of $\mathsf{ZF} + \mathsf{DC}$.

Now let $\mathsf{LM}$ be the statement ``all subsets of $\mathbb{R}$ are Lebesgue measurable.''  It is well known that $\{0,1\}^{\mathbb{N}}$ with the coin-flip measure $\mu$ is isomorphic as a measure space to $[0,1]$ with Lebesgue measure.  So $\mathsf{LM}$ implies that all subsets of $\{0,1\}^{\mathbb{N}}$ are $\mu$-measurable, and therefore so are all functions $f : \{0,1\}^{\mathbb{N}} \to \{0,1\}$.  By Theorem \ref{main}, we have in $\mathsf{ZF} + \mathsf{DC}$ that $\mathsf{LM}$ implies ``all hat-guessing strategies (for two colors and countably many players) are only reasonably successful''.

This is all very well, but meaningless unless $\mathsf{ZF} + \mathsf{DC} + \mathsf{LM}$ is actually consistent (free of contradictions); an inconsistent axiomatic system proves \emph{every} statement (as well as its negation).  G\"odel's second incompleteness theorem tells us that we cannot hope to prove its consistency within $\mathsf{ZF}$ --- unless, again, $\mathsf{ZF}$ itself should happen to be inconsistent.

We might hope instead to show that  $\mathsf{ZF} + \mathsf{DC} + \mathsf{LM}$ is \emph{equiconsistent} with $\mathsf{ZF}$, so that any contradictions arising in the former must already have been present in the latter.  As we all believe, or at least dearly hope, that $\mathsf{ZF}$ is in fact consistent, this would be quite satisfactory.  Now, it turns out this is not quite true.  But what is true, by a famous result of Solovay \cite{Solovay1970} improved by Shelah \cite{shelah-solovays-inaccessible},  is that $\mathsf{ZF} + \mathsf{DC} + \mathsf{LM}$ is equiconsistent with $\mathsf{ZFC} + \mathsf{IC}$, where $\mathsf{IC}$ is ``an inaccessible cardinal exists.''  This is still pretty good, as $\mathsf{ZFC} + \mathsf{IC}$ is also generally believed to be consistent; from now on, let us presume that it is.

Putting this all together, we conclude that the existence of unreasonably successful strategies ($\mathsf{USS}$) \emph{cannot} be proved in $\mathsf{ZF}$, nor even in $\mathsf{ZF} + \mathsf{DC}$.  For if it could, then since by Theorem \ref{main} $\mathsf{USS}$ contradicts $\mathsf{LM}$, we would have a contradiction in $\mathsf{ZF} + \mathsf{DC} + \mathsf{LM}$.  Solovay's result would then imply that there is also a contradiction in $\mathsf{ZFC} + \mathsf{IC}$ --- but we believe it to be consistent.

\begin{remark}
  In \cite{ht-intel}, the authors carried out a similar analysis, but instead of the statement $\mathsf{LM}$ that every subset of $\mathbb{R}$ is Lebesgue measurable, they focused on the statement $\mathsf{BP}$ that every subset of $\mathbb{R}$ has the property of Baire (i.e., can be written as the symmetric difference of an open set and a meager set).  This has a slight advantage in that $\mathsf{ZF} + \mathsf{DC} + \mathsf{BP}$ is actually equiconsistent with $\mathsf{ZF}$ \cite{Solovay1970, shelah-solovays-inaccessible}, so we need only assume the consistency of $\mathsf{ZF}$ rather than of $\mathsf{ZFC} + \mathsf{IC}$ as above.  They showed that the existence of the Gabay--O'Connor and Lenstra strategies each imply the existence of a set without the property of Baire, and so, as above, $\mathsf{ZF} + \mathsf{DC}$ is unable to prove the existence of either one of them.  However, this still left open the question of whether slight weakenings of these strategies might be possible to produce in $\mathsf{ZF} + \mathsf{DC}$, which our results here help to address.
\end{remark}

Having shown that $\mathsf{ZF} + \mathsf{DC}$ cannot prove the existence of unreasonably successful strategies, whereas $\mathsf{ZFC}$ can, we see that all such proofs must necessarily make use of $\mathsf{AC}$ in a strong way, beyond what $\mathsf{DC}$ provides.  In this sense, in $\mathsf{ZFC}$, unreasonably successful strategies are examples of what Schechter in \cite{schechter} calls ``intangibles'': mathematical objects which, although their existence can be proved, cannot be constructed in any explicit manner.

As another way to understand this phenomenon, we could decide that we will \emph{not} work in $\mathsf{ZFC}$, but will deliberately reject $\mathsf{AC}$ as an axiom.  Instead, we might adopt $\mathsf{ZF} + \mathsf{DC} + \mathsf{LM}$ as our axioms of set theory.  This axiomatic system is, as we mentioned, adequate for doing general mathematics, and by Solovay's result, is consistent (with appropriate caveats).  Within this alternate mathematical universe, Theorem \ref{main} is still a theorem, but its measurability hypothesis becomes redundant, and it simply states ``Unreasonably successful strategies do not exist''.  Thus, this universe is, at least in this one respect, more in keeping with our intuition about hat guessing than the usual $\mathsf{ZFC}$ universe, which could be a reason to prefer it.  (There are others as well. For instance, in the $\mathsf{ZF} + \mathsf{DC} + \mathsf{LM}$ universe, there is no Banach--Tarski paradox.)

Along similar lines, we mention \cite[Chapter 27]{schechter}, and references therein, which looks at a ``dream universe'' satisfying $\mathsf{ZF} + \mathsf{DC} + \mathsf{BP}$.  By the results of \cite{ht-intel}, this universe also excludes the Gabay--O'Connor and Lenstra strategies.  It has many other appealing theorems besides; for instance,  that every linear mapping between Banach spaces is continuous.

(For a contrary view, however, see \cite{taylor-wagon-paradox}, in which the authors point out that $\mathsf{ZF} + \mathsf{DC} + \mathsf{LM}$ has some counter-intuitive consequences of its own.  For instance, following the familiar Vitali set construction, consider the equivalence relation $\sim$ on $\mathbb{R}$ where $x \sim y$ iff $x - y \in \mathbb{Q}$, and let $\mathbb{R} / \mathbb{Q}$ be the set of equivalence classes.  They show that $\mathsf{ZF} + \mathsf{DC} + \mathsf{LM}$ proves that there is an injection from $\mathbb{R}$ to $\mathbb{R} / \mathbb{Q}$, but none from $\mathbb{R} / \mathbb{Q}$ to $\mathbb{R}$.  That is, there are ``strictly more'' equivalence classes than there are elements of $\mathbb{R}$ itself.  They prove some further results along these lines, and use this as the basis of a spirited argument \emph{in favor} of the full axiom of choice.)

\section{Some explicit strategies for countably many players}\label{sec:explicit-strategies}

In this section, we explore some explicit measurable strategies for the hat game with countably many players, that illustrate what can still be achieved within the bounds of Theorem \ref{main}.  We continue the notation of Section \ref{sec:meas}.  In particular, $L$ and $U$ are respectively the lower and upper asymptotic density of the set of players who guess correctly, viewed as a random subset of $\mathbb{N}$.

\begin{proposition}
  There is a measurable strategy for which $L = U = \frac{1}{2}$ almost surely.
\end{proposition}

\begin{proof}
  The ``pairs'' strategy of Definition \ref{pairs} achieves this.  Indeed, by the strong law of large numbers, so does the trivial strategy under which every player simply guesses ``black.''
\end{proof}

\begin{proposition}\label{L0U1}
  There is a measurable strategy for which $L=0$ and $U=1$ almost surely.
\end{proposition}

\begin{proof}
  Fix a strictly increasing sequence of positive integers $0 = n_0 < n_1 < n_2 < \dots$, growing sufficiently rapidly so that  $\lim_{k \to \infty} \frac{n_{k-1}}{n_k} = 0$.  For instance, $n_k = k!$ for $k \ge 1$ would do.  Partition the players into consecutive blocks $B_k = \{ n_{k-1} + 1, \dots, n_k\}$.  Within each block, let the players follow the even-odd strategy of Definition \ref{even-odd}, ignoring all the other blocks, so each block is either entirely correct or entirely incorrect.

  Let $A_k$ be the event that everyone in block $k$ guesses correctly.  Since players' guesses are based only on the hats within their own block, the events $A_k$ are independent, and as noted in Proposition \ref{even-odd-result}, we have $\mathbb{P}(A_k) = \frac{1}{2}$.   By the second Borel--Cantelli lemma, we have $\mathbb{P}(A_k \text{ i.o.}) = 1$: almost surely, infinitely many blocks are correct.  Now on the event $A_k$, among the players numbered $1, \dots, n_k$, at most $n_{k-1}$ of them have guessed incorrectly, and so $\overline{Z}_{n_k} \ge 1-\frac{n_{k-1}}{n_k}$.  So on the almost sure event $\{A_k \text{ i.o.}\}$, we have $\overline{Z}_{n_k} \ge 1- \frac{n_{k-1}}{n_k}$ infinitely often, implying that $U = \limsup_{n \to \infty} \overline{Z}_n = 1$.

By the same argument, almost surely, infinitely many blocks are incorrect; and we similarly conclude that under this strategy, $L = 0$ almost surely.
\end{proof}

\begin{proposition}\label{LhalfUu}
  Fix any $\frac{1}{2} \le u \le 1$.  There is a measurable strategy for which $L = \frac{1}{2}$ and $U = u$ almost surely.
\end{proposition}

The basic idea of the strategy is similar to that of Proposition \ref{L0U1}: assign sufficiently large finite blocks of players to play the even-odd strategy, so that each block guesses entirely correctly or incorrectly.    When a block is entirely correct, the long run of correct guesses pushes $\overline{Z}_n$ higher, and if we can ensure this happens infinitely often with probability $1$, we achieve a higher value for $U$.  However, such a block of players may just as well all guess incorrectly, which tends to push $\overline{Z}_n$ lower, and could result in a lower value for $L$ than the $\frac{1}{2}$ we want in this proposition.  So here, we group together smaller blocks, who will coordinate their strategies so as to give an adequate chance of a long enough string of successes, yet without risking a long string of failures.  Achieving both goals turns out to require a somewhat delicate balancing act.

\begin{proof}
  The case $u = \frac{1}{2}$ is handled by the pairs strategy, so we assume $u > \frac{1}{2}$.
  
  We begin by describing the strategy generically.  Specific values for the sequences of parameters mentioned ($n_k$, $g_k$, etc) will be chosen later in the proof.
  
  We partition the players into \emph{teams} $T_k = \{n_k + 1, \dots, n_{k+1}\}$.  Within each team, the first $g_k$ players are designated as the \emph{gambler} squad $G_k = \{ n_k + 1, \dots, n_k + g_k \}$, and the remaining $r_k$ are the \emph{recovery} squad $R_k = \{n_k + g_k + 1, \dots, n_k + g_k + r_k = n_{k+1}\}$.  The gambler squad is further subdivided into $b_k$ blocks of equal size $s_k$, so that $g_k = b_k s_k$, and for $1 \le j \le b_k$ the $j$th block of team $k$ is $B_{k,j} = \{n_k + (j-1) s_k + 1, \dots, n_k + j s_k\}$.

  Now, the strategy is as follows.  The teams all guess independently of each other, and do not need to look at the hats of any other teams.  The gamblers in block $B_{k,j}$ are to look at the hats of the preceding blocks $B_{k,1}, \dots, B_{k,j-1}$ of their team.  If each of those $j-1$ blocks contained an even number of white hats, so that under the even-odd strategy they would all guess correctly, then block $B_{k,j}$ plays the even-odd strategy within their block; otherwise, they play the pairs strategy.  The recovery squad $R_k$ plays the pairs strategy unconditionally, which will have the effect of pushing $\overline{Z}_n$ back toward $\frac{1}{2}$.

  Let $A_k$ be the event that every one of the $b_k$ gambler blocks of team $T_k$ had an even number of white hats.  Clearly $\mathbb{P}(A_k) = 2^{-b_k}$, and the events $A_k$ are mutually independent (since each team $T_k$ made their guesses based solely on the hats of their own team).  On this event, every block played the even-odd strategy, and they all won.  So the results for team $T_k$ are that the first $g_k$ members all guessed correctly, and the remaining $r_k$ played pairs.
  
  On $A_k^c$, let $M_k$ be the index of the first gambler block of team $T_k$ that had an odd number of white hats.  Then blocks $1, \dots, M_k-1$ had an even number of white hats, played even-odd, and won; block $M_k$ played even-odd and lost; and blocks $M_k+1, \dots, b_k$, as well as the recovery squad, all played pairs.  So in all cases, the results of team $T_k$ consist of some number of correct guesses (possibly zero), possibly followed by $s_k$ incorrect guesses, and all remaining players (at least $r_k$ of them) play pairs.

  Let us now choose the sizes of these various groups.  We want to ensure that $\mathbb{P}(A_k \text{ i.o.}) = 1$, so by the second Borel--Cantelli lemma, it suffices to have $\sum_{k=1}^\infty \mathbb{P}(A_k) = \sum_{k=1}^\infty 2^{-b_k} = +\infty$.  On the other hand, to ensure the upper density converges as desired to $u$, the number of blocks $b_k$ must increase ($b_k \to +\infty$), so that the fluctuations above and below the value of $u$ become small. We achieve both goals by setting $b_k = \lfloor \lg (k+2) \rfloor$, where adding 2 simply ensures that $b_k > 0$ for all $k \ge 0$.  (Here $\lg x = \log_2 x$ denotes the base 2 logarithm.)

  Fix a sequence $u_k \to u$ with $\frac{1}{2} < u_k < u$, and for later convenience, let the $u_k$ all be rational.  Moreover, in case $u=1$, let the $u_k$ be chosen so that $(1-u_k) b_k \to \infty$; for instance, we could take $u_k \approx 1 - \frac{1}{\sqrt{b_k}}$.  Set $\alpha_k = \frac{u_k - \frac{1}{2}}{1 - u_k}$, which is also rational and satisfies $\frac{\alpha_k}{b_k} \to 0$, and $\alpha_0 = 0$.  Also, fix any sequence $\epsilon_k > 0$ with $\epsilon_k \to 0$.
  
  Now construct $n_k, g_k, r_k$ recursively.  We will ensure that at each stage, $n_k \frac{\alpha_k}{b_k}$ is an even integer.  Begin with $n_0 = 0$, and set:
  \begin{itemize}
  \item $g_k = \alpha_k n_k$
  \item $r_k$ an even integer sufficiently large that
    \begin{equation*}
      \frac{1}{2} - \epsilon_{k+1} < \frac{\frac{1}{2} r_k}{n_k + g_k + r_k} \le \frac{\frac{1}{2} r_k + n_k + g_k}{n_k + g_k + r_k} < \frac{1}{2} + \epsilon_{k+1}
    \end{equation*}
    and chosen larger if needed so that $(n_k + g_k + r_k)\frac{\alpha_{k+1}}{b_{k+1}}$ is an even integer.  
  \item $n_{k+1} = n_k + g_k + r_k$.
  \end{itemize}
  Now each gambler block of team $k$ is of size $s_k = \frac{g_k}{b_k}$ which by construction is an even integer.

  We begin our proof by noting that, of the first $n_{k+1}$ players, the remainder squad $R_k$ of team $T_k$ was playing pairs, so exactly half of those $r_k$ players guessed correctly.  Without paying any attention to what the remaining $n_k + g_k$ players did, we know that the fraction of correct guesses $\overline{Z}_{n_{k+1}}$ among players $1, \dots, n_{k+1}$ is certainly between $\frac{\frac{1}{2} r_k}{n_{k+1}}$ and $\frac{\frac{1}{2} r_k + n_k + g_k}{n_{k+1}}$. So by our choice of $r_k$ and re-indexing, we have
  \begin{equation}\label{nk-bound}
    \frac{1}{2} - \epsilon_{k} < \overline{Z}_{n_k} < \frac{1}{2} + \epsilon_{k}, \qquad k \ge 1.
  \end{equation}

  Now we show this strategy gives the desired values for $L$ and $U$.  As we noted above, regardless of hat assignment, the results of team $T_k$ consist of some number of correct guesses (possibly zero), possibly followed by $s_k$ incorrect guesses, followed by pairs players.  The correct guesses will cause $\overline{Z}_i$ to increase, and the pairs players will cause it to approach $\frac{1}{2}$ (monotonically, if we restrict our attention to even $i$, which involves no loss of generality).  The $s_k$ potential incorrect guesses will affect $\overline{Z}_i$ the most if they come before any correct guesses.  So the least possible value for $\overline{Z}_i$ among $n_k \le i \le n_{k+1}$ occurs at $i_0 = n_k + s_k$ if the first gambler block has an odd number of white hats and guesses wrong.  This implies that for all $n_k \le i \le n_{k+1}$ we have
  \begin{equation*}
    \overline{Z}_i \ge \frac{n_k \overline{Z}_{n_k}}{n_k + s_k} = \overline{Z}_{n_k} - \frac{s_k}{n_k + s_k} \overline{Z}_{n_k} \ge \left(\frac{1}{2} - \epsilon_k\right) - \frac{s_k}{n_k}
  \end{equation*}
  using \eqref{nk-bound} in the first term, and $\frac{s_k}{n_k+s_k} \le \frac{s_k}{n_k}$ and $\overline{Z}_{n_k} \le 1$ in the second.  Now we have $\epsilon_k \to 0$, and also $\frac{s_k}{n_k} = \frac{\alpha_k}{b_k} \to 0$.  Since this inequality holds for all $k$, we have $L \ge \frac{1}{2}$.  As we already know $L \le \frac{1}{2}$ by Theorem \ref{main}, we have shown $L = \frac{1}{2}$.

  Next, note that the greatest possible value for $\overline{Z}_i$ occurs at $i_0 = n_k + g_k$, if all gambler blocks guessed correctly (the event $A_k$).  This gives
  \begin{equation*}
    \overline{Z}_i \le \frac{n_k \overline{Z}_{n_k} + g_k}{n_k + g_k} = \frac{\overline{Z}_{n_k} + \alpha_k}{1 + \alpha_k} \le
\frac{\left(\frac{1}{2} + \epsilon_k\right) + \alpha_k}{1 + \alpha_k}
 =    u_k + \frac{\epsilon_k}{1 + \alpha_k}.
  \end{equation*}
  Now $u_k \to u$, and $\frac{\epsilon_k}{1 + \alpha_k} < \epsilon_k \to 0$.  Thus we have $U \le u$.

  When $A_k$ does in fact happen, we have
  \begin{equation*}
     \overline{Z}_{n_k + g_k} = \frac{n_k \overline{Z}_{n_k} + g_k}{n_k + g_k} = \frac{\overline{Z}_{n_k} + \alpha_k}{1 + \alpha_k} \ge \frac{\left(\frac{1}{2} - \epsilon_k\right) + \alpha_k}{1 + \alpha_k} = u_k - \frac{\epsilon_k}{1 + \alpha_k}.
  \end{equation*}
 So on the event $\{A_k \text{ i.o.}\}$, which as we noted has probability $1$, we have $\overline{Z}_{n_k + g_k} \ge u_k - \frac{\epsilon_k}{1 + \alpha_k}$ infinitely often.  This implies
  \begin{equation*}
    U \ge \limsup_{k \to \infty} \left[ u_k - \frac{\epsilon_k}{1 + \alpha_k} \right] = u
  \end{equation*}
  almost surely.
\end{proof}

\begin{corollary}\label{cor:arbitrary-L-U}
  For any pair of values $0 \le \ell \le \frac{1}{2} \le u \le 1$ allowed by Theorem \ref{main}, there is a measurable strategy such that $L = \ell$ and $U = u$ almost surely.  
\end{corollary}

\begin{proof}
  If we make the gamblers ``play to lose'', so that block $B_{k,j}$ plays even-odd if every one of the blocks $B_{k,1}, \dots, B_{k, j-1}$ had an \emph{odd} number of white hats, we can get a strategy where $L = \ell$ and $U = \frac{1}{2}$ almost surely, for any $0 \le \ell \le \frac{1}{2}$.  Then if we alternate ``play to win'' and ``play to lose'' teams, we finally can obtain a strategy with $L = \ell$ and $U = u$ for any $0 \le \ell \le \frac{1}{2} \le u \le 1$.
\end{proof}

Mixed strategies are also possible.

\begin{corollary}
  For any probability distribution $\mu$ on $[0, \frac{1}{2}] \times [\frac{1}{2}, 1]$, there is a measurable strategy such that $(L,U) \sim \mu$.  That is, every \emph{joint distribution} for $(L,U)$ allowed by Theorem \ref{main} can actually be attained.
\end{corollary}

\begin{proof}
  Identify an infinite set $D = \{d_1, d_2, \dots\} \subseteq \mathbb{N}$ with asymptotic density zero (e.g. the square-numbered players), and designate those players as ``inactive''.  Their only role is to have their hat colors serve as a noise source, so they can guess whatever they want; being a set of density zero, their guesses cannot affect the values of $L$ or $U$.  The remaining ``active'' players re-index themselves by the positive integers, as if they were the only ones playing.  Now for any probability distribution $\mu$ on $[0, \frac{1}{2}] \times [\frac{1}{2}, 1]$ (or any other standard Borel space), we can find a measurable function $F : \{0,1\}^{\mathbb{N}} \to [0, \frac{1}{2}] \times [\frac{1}{2}, 1]$ such that $F(\xi_{d_1}, \xi_{d_2}, \dots) \sim \mu$.  The players evaluate $F(\xi_{d_1}, \xi_{d_2}, \dots) = (L_0, U_0)$, and then conditional on this, they execute a strategy as above which results in $L = L_0$ and $U = U_0$ almost surely.  By this approach, we have a mixed strategy for which $(L,U) \sim \mu$.
\end{proof}

In a different direction, following a suggestion by Elliot Glazer and Charles Wang (personal communication), one can study the \emph{absolute} number of ``excess'' correct guesses among the first $n$ players, $S_n - \frac{n}{2}$.  Theorem \ref{main} shows that, for measurable strategies, $S_n - \frac{n}{2}$ cannot grow at a linear rate with positive probability, but it does not rule out the possibility of a measurable strategy for which $S_n - \frac{n}{2} \to +\infty$ almost surely at a sub-linear rate.  In fact, this can be achieved.

\begin{proposition}\label{absolute-to-infinity}
  There is a measurable strategy under which $S_n - \frac{n}{2} \to +\infty$ almost surely.
\end{proposition}

\begin{proof}
  Divide the players into consecutive blocks of (even) sizes $t_k$, to be chosen later.  The players in each block only need to look at their own block's hats, so the blocks play independently of one another.

  Each block normally plays the pairs strategy, where each pair consists of an ``S'' player who guesses their own hat is the same as their partner, and a ``D'' player who guesses that their hat is different.  However, we introduce the following exception: if any player sees only black hats within their block, then that player is to abandon the pairs strategy and guess ``white''.   We have the following cases for each block:
  \begin{itemize}
  \item Bad case: all hats in the block are black. Then every player sees all black hats and guesses ``white,'' and they are all wrong. 
  \item Good case: all hats are black except one, and the one white hat belongs to an S player.  Suppose their name is Sam.  Then Sam sees all black hats, so following the exception, Sam guesses ``white'' and is correct.  Sam's partner (a D player) sees that Sam has a white hat, so following their normal D strategy, they guess ``black'' and are also correct.  In each remaining partnership, the normal pairs strategy is followed and exactly one player guesses right.  So the number of correct guesses in the block is $\frac{t_k}{2} + 1$.
  \item Neutral case: all other possibilities.  If at least two hats are white, then everyone sees at least one white hat and follows their normal pairs strategy.  Otherwise, exactly one hat is white and it belongs to an D player; call them Danny.  Danny guesses ``white'' according to the exception, but since Danny's partner has a black hat, this is what Danny was going to guess anyway.  So in either case, exactly $\frac{t_k}{2}$ of the block's players guess right.
  \end{itemize}
  Let $B_k$ be the event that block $k$ is bad, and $G_k$ the event that it is good.  It is easy to see that $\mathbb{P}(B_k) = 2^{-t_k}$, and $\mathbb{P}(G_k) = \frac{t_k}{2} \cdot 2^{-t_k}$.

  Now we would like to arrange that, almost surely, only finitely many blocks are bad, and infinitely many blocks are good.  Since every good block increases $S_n - \frac{n}{2}$ by $1$, this will ensure $S_n - \frac{n}{2} \to +\infty$ a.s.  (Also note that within a good or neutral block, at least one of every pair guesses correctly, so $S_n - \frac{n}{2}$ is nondecreasing within such a block if we consider only even values for $n$.)

  By the Borel--Cantelli lemmas, this will hold provided that $\sum_k \mathbb{P}(B_k) = \sum_k 2^{-t_k} < \infty$ and (using the independence of blocks) $\sum_k \mathbb{P}(G_k) = \sum_k \frac{t_k}{2} 2^{-t_k} = \infty$.  There are various choices for $t_k$ that will achieve this; for instance, one may take
  \begin{equation*}
    t_k = \lg (k \lg^2 k) = \lg k + 2 \lg \lg k
  \end{equation*}
  (disregarding $k=0,1$ and rounding to the nearest even integer).  The convergence and divergence of the relevant sums can be confirmed with the integral test.
\end{proof}

\begin{remark}
It may seem at first glance that Proposition \ref{absolute-to-infinity} would contradict the ``symmetry'' of the game, and in particular the fact that $\mathbb{E}[S_n] = \frac{n}{2}$ for all $n$.  But one can note in the strategy that although the number of bad blocks, and hence the total number of players they contain, is a.s. finite, yet the \emph{expected} number of players in bad blocks is infinite.  So the finitely many extra incorrect guesses from the bad blocks actually do still balance with the infinitely many extra correct guesses from the good blocks, ``on average.''
\end{remark}

\begin{remark}
  It can be shown that, for the strategy in the proof of Proposition \ref{absolute-to-infinity}, taking $t_k = \lg (k \lg^2 k)$, we in fact have $S_n - \frac{n}{2} \asymp \log \log n$ almost surely.  (By this we mean that there are random variables $0 < K_1 < K_2 < \infty$ and $1 \le N_0 < \infty$ such that $K_1 \log \log n \le S_n - \frac{n}{2} \le K_2 \log \log n$ for all $n \ge N_0$, almost surely.)  
  This can be sped up somewhat by choosing a slightly slower-growing $t_k$; for instance, taking $t_k = \lg\left(k \lg k \cdot (\lg \lg k)^2\right)$, we get $S_n - \frac{n}{2} \asymp \frac{\log n}{(\log \log n)^2}$.  However, variations on this theme do \emph{not} attain $S_n - \frac{n}{2} \asymp \log n$.  This raises the following interesting questions, to which we do not know the answer:
  \begin{enumerate}
  \item     Does there exist a measurable strategy under which we have $S_n - \frac{n}{2} \asymp \log n$ almost surely, or even some faster rate?
  \item Can one obtain a universal upper bound on the lower growth rate of $S_n - \frac{n}{2}$, over all measurable strategies, stronger than  that of Theorem \ref{main}?  That is, to find a function $g(n)$ such that for every measurable strategy we have $\liminf_{n \to \infty} \frac{1}{g(n)} (S_n - \frac{n}{2}) = 0$ almost surely.  Theorem \ref{main} shows this holds for $g(n)=n$, but from the discussion above, it is conceivable that it could hold for a rate as slow as $g(n) = \log n$.
  \end{enumerate}
\end{remark}

\section{Variants of the hat-guessing game}\label{sec:variants}

\subsection{More hat colors}

Suppose that instead of two hat colors, we have some larger set $C$ of possible hat colors, which may be finite or infinite.   Intuitively, this seems to make the game harder to win.  Yet in the $\mathsf{AC}$ setting with infinitely many players, the Gabay--O'Connor strategy still works exactly as before, guaranteeing that all but finitely many players guess correctly!  The Lenstra strategy can also be adapted by replacing $\mathbb{F}_2$ by an abelian group of the same cardinality as $C$ (under $\mathsf{AC}$ such a group exists for every cardinality, by the upward L\"owenheim--Skolem theorem \cite{hajnal-kertesz}). As before, it guarantees that either all players guess right or all players guess wrong.  

For finitely many players and finitely many colors $K=|C|$, we have analogues of the ``even-odd'' and ``pairs'' strategy, where instead of using parity, we index the hat colors by the integers $\{0, \dots, K-1\}$, and think about sums of the hat colors mod $K$ \cite{feige}. This gives us a strategy where all players guess right with probability $\frac{1}{K}$ (and otherwise all guess wrong), and a strategy that guarantees that exactly $\frac{1}{K}$ of the players guess right.  Using as before the fact that every player's guess is independent of their own hat color, we get that the expected fraction of players guessing right is $\frac{1}{K}$, and so both strategies are best possible.

Now we turn to the case of most interest to us, where there are countably many players and only measurable strategies are allowed.  If the number of hat colors $K$ is finite, one can easily adapt the proof of Theorem \ref{main} to show that almost surely, $0 \le L \le \frac{1}{K} \le U \le 1$.

If countably many hat colors are available, then it can be arbitrarily unlikely for even one player to guess correctly.  Indeed, fix $\epsilon > 0$, and choose a sequence of integers $K_i$ for which $\sum_{i=1}^\infty \frac{1}{K_i} < \epsilon$ (for instance, take $K_i > 2^i/\epsilon$).  Assign hats to players independently, such that the hat color of player $i$ is chosen uniformly from the colors $\{0, \dots, K_i-1\}$.  Let us generously allow the players to know the values of the integers $K_i$, so that each one always guesses a color in the appropriate range.  (If not, the probability of correct guessing would only decrease.)  Then, for any measurable strategy, as before, player $i$ has probability $\frac{1}{K_i}$ of guessing correctly.  By a union bound, the probability that at least one player guesses correctly is bounded above by $\sum_{i=1}^\infty \frac{1}{K_i} < \epsilon$.  This is best possible, since if a player makes their guess according to any  distribution giving positive probability to each nonnegative integer (say, the Poisson distribution), one can show that they have a nonzero probability of being correct, no matter how the hat colors are selected.

Thus, with infinitely many hat colors, there can be no measurable strategy that guarantees at least one correct guess (a \emph{minimal predictor} in the terminology of Hardin and Taylor; see \cite{ht-book,ht-minimal}).  So the success of the Gabay--O'Connor strategy seems even more unreasonable in this setting.  

Beyond this, consider the case of continuum many hat colors; say, $C=[0,1]$. If each player is independently assigned a hat chosen according to the uniform distribution on $[0,1]$ (or any other continuous or atomless distribution), then we can show that for any measurable strategy, each player guesses incorrectly with probability $1$.  Taking a countable intersection of these almost sure events, we see that almost surely, \emph{all} players guess incorrectly.  In contrast, the Gabay--O'Connor strategy works just as well as before, which now seems to be passing from ``unreasonable'' to ``preposterous''.

\subsection{More players}

One could imagine playing the hat-guessing game with an uncountable set of players.  We have not attempted to explore this case from a probability perspective, since the model of assigning hat colors in an iid fashion behaves poorly, owing to the deficiencies of uncountable product measures, and no other model seems as natural.  We note, however, that the Gabay--O'Connor and Lenstra strategies \emph{still} work exactly as before, even if the set of players \emph{and} the set of hat colors are both of arbitrarily high cardinalities.  ``Preposterous'' does not seem like a strong enough word for this phenomenon!

\subsection{Limited visibility and other variants}\label{sec:other-variants}\label{sec:limited-vis}

Much of the literature on hat-guessing games involves ``limited visibility'' variants, where each player may only be able to see some subset of the other players' hats.  For instance, the players might stand in a line, with each player able to see only those players behind them.  They might also allow for some limited communication among the players after hats are assigned; for instance, the players might be asked to guess sequentially, with each player hearing the guesses of all preceding players.  Or, we might allow some players to ``pass'' and refrain from guessing, or consider different victory conditions.

For variant hat-guessing games involving infinitely many players, the axiom of choice is again often relevant, allowing strategies that are far more effective than without it.  We refer to \cite{
  ht-book,
  ht-minimal,
  glr,
  horsten-aftermath,
  Velleman2012,
  hamkins-blog,
  serafin}
where many such variants are posed and analyzed.

There is a much larger literature on variant hat-guessing games for finitely many players.  We cannot give an exhaustive list, but would mention in particular
\cite{
  ht-book,
  butler-et-al,
  pratt-wagon-wiener,
  buhler,
  ebert-thesis,
  feige,
  winkler,
  krzywkowski-survey,
  hamkins-blog,
  a-dozen-hat-problems,
  gnome-formula,
  paterson-stinson,
  more-colorful-hat,
  hat-problem-and-some-variations,
  ebert-merkle-vollmer,
  line-of-sages,
  bicycle-or-unicycle}.

\bibliographystyle{plainnat}
\bibliography{hats}

\end{document}